\documentclass[a4paper, 11pt]{article}
\usepackage[margin=2cm]{geometry}

\title{Dynamical properties of disjunctive Boolean networks} 

\author{Maximilien Gadouleau}

\usepackage{amssymb,amsmath,amsthm,multirow,xcolor,cancel,bbm,tikz,mathrsfs,calc}
\usepackage{subfig}

\newcommand{\Z}{\mathbb{Z}}

\newcommand{\B}{\mathbb{B}}

\newcommand{\inNeighbourhood}{N^{\mathrm{in}}}
\newcommand{\outNeighbourhood}{N^{\mathrm{out}}}
\newcommand{\InteractionGraph}{\mathbb{D}}
\newcommand{\wH}{w_\mathrm{H}}
\newcommand{\Define}[1]{\textbf{#1}}
\newcommand{\supp}{\mathrm{supp}}
\newcommand{\dH}{d_\mathrm{H}}

\newcommand{\dc}{d_C}
\newcommand{\GF}{\mathrm{GF}}
\newcommand{\Topology}{\mathbb{T}}

\newtheorem{proposition}{Proposition}
\newtheorem{lemma}{Lemma}
\newtheorem{theorem}{Theorem}
\newtheorem{corollary}{Corollary}

\newtheorem{problem}{Problem}

\begin{document}
	
	\maketitle

\begin{abstract}
A Boolean network is a mapping $f :\{0,1\}^n \to \{0,1\}^n$, which can be used to model networks of $n$ interacting entities, each having a local Boolean state that evolves over time according to a deterministic function of the current configuration of states. 
In this paper, we are interested in disjunctive networks, where each local function is simply the disjunction of a set of variables. As such, this network is somewhat homogeneous, though the number of variables may vary from entity to entity, thus yielding a generalised cellular automaton. The aim of this paper is to review some of the main results, derive some additional fundamental results, and highlight some open problems on the dynamics of disjunctive networks.
We first review the different defining characteristics of disjunctive networks and several ways of representing them using graphs, Boolean matrices, or binary relations.
We then focus on three dynamical properties of disjunctive networks: their image points, their periodic points, and their fixed points. For each class of points, we review how they can be characterised and study how many they could be. The paper finishes with different avenues for future work on the dynamics of disjunctive networks and how to generalise them.
\end{abstract}

\section{Introduction}
Consider a finite network of $n$ entities $1 \le i \le n$, where each has a Boolean state $x_i \in \{0,1\}$ that evolves over time according to a deterministic rule $f_i(x_1, \dots, x_n) : \{0,1\}^n \to \{0,1\}$. The evolution of the configuration of states $x = (x_1, \dots, x_n)$ is fully characterised by the global update function $f = (f_1, \dots, f_n) : \{0,1\}^n \to \{0,1\}^n$. Such a function is called a Boolean network. Boolean networks are a versatile model and as such have been used to represent different networks, such as gene networks, neural networks, social networks, or network coding (see \cite{GR16} and references therein for the applications of Boolean networks).

Cellular automata form a major model of discrete interactions. The network of entities (sometimes referred to as cells) is homogeneous, in the sense that the local update functions are all similar, and even though the number of cells is usually infinite, the interactions are only local. Cellular automata have attracted a huge amount of interest, both for their theoretical properties and their numerous applications (see \cite{HM17,Ila01}).

Boolean networks and Cellular automata share some common typical characteristics, such as discrete space, discrete time, and finite state values for each entity (though some generalisations do not share all those characteristics). However, they differ in many aspects. Boolean networks are very general, with heterogeneous network topology, different local functions, finite number of entities, and various update schedules (synchronous or asynchronous). Cellular automata are much more restricted, with regular topology (usually a lattice $\Z^d$), the same local update function everywhere, an infinite number of cells, and typically parallel updates.

There are many ways to possibly bridge the gap between these two models, by either relaxing some properties of CA or focusing on specific Boolean networks. In this paper, we are interested in a special class of Boolean networks where the local functions $f_i$ are ``the same'' everywhere, despite different cells interacting with more or less cells. In particular, we focus on disjunctive networks, where the local function at every entity is the disjunction of some variables:
\[
    f_i(x) = \bigvee_{j \in N(i)} x_j.
\]
The set of variables $N(i)$ does depend on the entity and its size may vary.

Since disjunctions are such special Boolean functions, disjunctive networks can be characterised in different ways. We first review these characterisations in Section \ref{section:characterisations}. Also, disjunctive networks can be represented using graphs, Boolean matrices, or binary relations; again, we review these representations in Section \ref{section:representations}.

The dynamical properties of Boolean networks have been thoroughly studied, see \cite{ARS14,ADG04,ARS17,GRR15,Gol85,GT83,PR12} for example. Due to their different representations, disjunctive networks have attracted interest inside and outside of the Boolean network community. For instance, some early work on binary relations and Boolean matrices classifies convergent and idempotent disjunctive networks \cite{Kim72,Kim82,Ros63}. The main interests in the dynamical properties of disjunctive networks include the transient length \cite{GH00}, the characterisation of their cycle structure \cite{GH00, SLV10}, and in particular determining when they present no oscillations \cite{AMPV12,SLV10}. It is worth noting that some results on disjunctive networks are included in works that consider related or more general classes of Boolean networks, e.g. \cite{ADG04a,ARS14,ADG04b,ARS17}. Even though in this paper we focus on the synchronous dynamics (i.e. parallel updates), the reader interested in the asynchronous dynamics of disjunctive networks is directed to \cite{GM12}.

In this paper, we are interested in the following three dynamical features of a disjunctive network. An image point of $f$ is a reachable state; a periodic point is a recurring state ; and a fixed point is a stationary state. We first give characterisations of their sets of image points, periodic points, and fixed points, respectively in Section \ref{subsection:points}. We then consider the number of image, periodic, and fixed points of disjunctive networks. We prove some results on the possible values these numbers can take in Section \ref{subsection:ranks}.

In this paper, we survey some of the main results on disjunctive networks, obtain some new results, and highlight some open problems in that area. The rest of this paper is organised as follows. Section \ref{section:elementary_properties} defines disjunctive networks, gives different characterisations of such networks, and illustrates how they can be represented. In Section \ref{section:dynamic_properties}, we study the dynamical properties of disjunctive networks when updated synchronously, with a focus on image points, periodic points, and fixed points. Finally, some possible avenues for generalisations and future work are given in Section \ref{sec:conclusion}.

\section{Elementary properties} \label{section:elementary_properties}

\subsection{Definition} \label{section:definition}

We denote the Boolean alphabet as $\B = \{0,1\}$, which we endow with the natural order $0 < 1$. For any $x,y \in \B$, their \Define{disjunction} is simply $x \lor y = \max \{ x, y \}$.  The disjunction operation satisfies the following properties (for all $x,y,z \in \B^n)$:
\begin{itemize}
    \item It is associative: $x \lor ( y \lor z ) = (x \lor y) \lor z$.
    
    \item It is commutative: $x \lor y = y \lor x$.
    
    \item It has an identity element, namely $0$: $x \lor 0 = 0 \lor x = x$.
\end{itemize}
In view of these properties, we can generalise disjunction to an arbitrary number of Boolean variables. Let $S$ be a finite set and consider a configuration $x = (x_s : s \in S) \in \B^S$, then 
\[
    \bigvee_{s \in S} x_s = \begin{cases}
        1 & \text{if } \exists t \in S : x_t = 1\\
        0 & \text{otherwise}.
    \end{cases}
\]
In particular, if $S = \emptyset$, then $\bigvee_{s \in S} x_s = 0$; if $S = \{s\}$, then $\bigvee_{s \in S} x_s = x_s$.

The \Define{conjunction} of $x,y \in \B$ is $x \land y = \min \{ x, y \}$. Since conjunction and disjunction are dual, i.e. they are equivalent up to re-ordering $0$ and $1$, all the results we shall state about disjunction can be translated to apply to conjunction as well. In particular, we can define the conjunction of an arbitrary set of variables. For our purposes, it will be useful to group disjunction and conjunction under one common name. As such, we say that an operation is a \Define{junction} if it is either a disjunction or a conjunction.

A point $x \in \B^n$ is called a \Define{configuration}, which we shall denote as $x = (x_1, \dots, x_n)$ where $x_i \in \B$ for all $i \in [n] := \{1, \dots, n\}$. We introduce some further notation for configurations: we write $x = (x_i, x_{-i})$ where $x_{-i} = (x_1, \dots, x_{i-1}, x_{i+1}, \dots, x_n) \in \B^{n-1}$, and we define the unit configuration $e^j$ to satisfy $e^j_j = 1$, $e^j_{-j} = 0_{-j}$. The Hamming distance between $x,y \in \B^n$ is the number of coordinates they disagree: $\dH( x, y )  = |\{ i \in [n] : x_i \ne y_i\}|$; the Hamming weight of $x$ is the number of ones in $x$: $\wH(x) = |\{ i : x_i = 1 \}|$.

A \Define{Boolean network} of dimension $n$ is a mapping $f : \B^n \to \B^n$. We also split $f$ as $f = (f_1, \dots, f_n)$ where $f_i : \B^n \to \B$ is a Boolean function representing the update of the state $x_i$ of the $i$-th entity of the network. A Boolean network $f$ is \Define{disjunctive} if $f_i$ is a disjunction for all $i \in [n]$.

A (directed) graph $D = (V,E)$ is a pair where $V$ is the set of vertices and $E \subseteq V^2$ is the set of arcs of $D$ \cite{BG09a}. In this paper, we only consider finite graphs and we shall usually identify isomorphic graphs. For any vertex $i \in V$, its in-neighbourhood in $D$ is $\inNeighbourhood(i) = \{ j \in V : (j,i) \in E \}$ and its in-degree is the size of its in-neighbourhood; the out-neighbourhood an out-degree are defined similarly. A vertex is a source (sink, respectively) if its in-neighbourhood (out-neighbourhood, respectively) is empty.

The \Define{interaction graph} of $f$, denoted as $\InteractionGraph(f)$, represents the influences of entities on one another. Formally, $\InteractionGraph(f) = (V,E)$, where $V = [n]$ and $(i,j) \in E$ if and only if $f_j$ depends essentially on $x_i$, i.e. there exists $a_{-i}$ such that
\[
    f_j(0,a_{-i}) \ne f_j(1, a_{-i}).
\]
If $D$ is the interaction graph of $f$, we then say that $f$ is a Boolean network on $D$. Clearly, for every graph $D$ there is a unique disjunctive network on $D$.

\subsection{Characterisations} \label{section:characterisations}

We can extend the order $0 < 1$ in $\B$ to configurations $x,y \in \B^n$ componentwise: we write $x \le y$ if and only if $x_i \le y_i$ for all $i \in [n]$. We can also extend the disjunction notation to configurations by applying it componentwise: $x \lor y = z$ with $z_i = x_i \lor y_i$ for all $i \in [n]$. We then have
\begin{equation} \label{equation:xley}
    x \le y \iff x \lor y = y.
\end{equation}

A Boolean function $\phi : \B^n \to \B$ is \Define{monotone} if $x \le y$ implies $\phi(x) \le \phi(y)$. It is easily checked that any junction is monotone. We say a Boolean network $f$ is \Define{monotone} if $x \le y$ implies $f(x) \le f(y)$; clearly $f$ is monotone if and only if $f_i$ is monotone for all $i$.

\begin{lemma} \cite[Theorem 11.1]{Hel11} \label{lemma:monotone}
A Boolean network $f$ is monotone if and only if for all $x,y \in \B^n$, 
\begin{equation} \label{equation:monotone}
    f(x \lor y) \ge f(x) \lor f(y).    
\end{equation}
\end{lemma}

\begin{proof}
Suppose $f : \B^n \to \B^n$ is monotone, then $f(x \lor y) \ge f(x)$ and $f(x \lor y) \ge f(y)$, hence $f(x \lor y) \ge f(x) \lor f(y)$. Conversely, if $f(x \lor y) \ge f(x)$ for all $x,y$, then for any $a \le b$ we have $f(b) = f(a \lor b) \ge f(a)$, hence $f$ is monotone. 
\end{proof}

Combining \eqref{equation:xley} and \eqref{equation:monotone}, we obtain a characterisation of monotone networks based on an equation.

\begin{corollary} \cite[Example 11.5]{Hel11}
A Boolean network $f$ is monotone if and only if for all $x,y \in \B^n$,
\[
    f(x \lor y) = f(x) \lor f(x \lor y).
\]
\end{corollary}

The first, and arguably canonical, characterisation of disjunctive networks is that they are the endomorphisms of the disjunction on $\B^n$. As such, disjunctive networks are those that reach equality in \eqref{equation:monotone} and that fix the all-zero configuration. The latter property is a technical detail, which comes the fact that the constant Boolean function $\phi(x) = 1$ is not a disjunction and yet also satisfies $\phi(x \lor y) = \phi(x) \lor \phi(y)$.

\begin{theorem} \label{theorem:characterisation1}
A Boolean network $f$ is disjunctive if and only if $f(0, \dots, 0) = (0, \dots, 0)$ and for all $x,y \in \B^n$,
\[
    f(x \lor y) = f(x) \lor f(y).
\]
\end{theorem}

\begin{proof}
The forward implication is straightforward; we focus on the reverse implication. Suppose $f$ fixes $(0,\dots,0)$ and $f(x \lor y) = f(x) \lor f(y)$ for all $x,y \in \B^n$. First, by Lemma \ref{lemma:monotone}, $f$ is monotone. If $f$ is not disjunctive, then $f_i$ is not a disjunction for some $i$, and hence by monotonicity there exists $j \in \inNeighbourhood(i)$ such that $f_i(e^j) = 0$. There exists $a_{-j}$ such that $f_i(0,a_{-j}) = 0$ and $f_i( 1, a_{-j} ) = 1$. We obtain
\begin{align*}
    f_i( (0, a_{-j}) \lor e^j ) = f_i(1, a_{-j}) &= 1,\\
    f_i( 0, a_{-j} ) \lor f_i(e^j) = 0 \lor 0 &= 0,
\end{align*}
which is the desired contradiction.
\end{proof}

The second characterisation is that disjunctive networks are precisely the submodular monotone networks. A Boolean network is \Define{submodular} if for all $x,y \in \B^n$,
\[
    f(x \lor y) \lor f(x \land y) \le f(x) \lor f(y).
\]
Submodular Boolean functions form an important class of Horn functions; the interested reader is directed to \cite[Section 6.9.1]{Bor11} and references therein. Theorem \ref{theorem:characterisation1} then immediately yields a second characterisation of disjunctive networks, implicit from Theorems 11.1 and 11.4 in \cite{Hel11}.

\begin{theorem}  \label{theorem:characterisation2}
A Boolean network is disjunctive if and only if it is monotone and submodular and it fixes the all-zero configuration.
\end{theorem}

The third characterisation of disjunctive networks is that of maximal locally idempotent networks. A Boolean function $\phi : \B^n \to \B$ is \Define{idempotent} if $\phi(0,\dots, 0) = 0$ and $\phi(1, \dots, 1) = 1$ \cite{PR10a}. We then say a Boolean network $f$ is \Define{locally idempotent} if $f_i$ is idempotent for all $i$; equivalently $f$ is locally idempotent if it fixes the all-zero and the all-one configurations. Since an idempotent Boolean function is not constant, the interaction graph of a locally idempotent network has no sources. Conversely, any monotone network on a graph with no sources is locally idempotent. For two Boolean networks $f,g : \B^n \to \B^n$, we naturally write $f \le g$ if $f(x) \le g(x)$ for all $x \in \B^n$.

\begin{theorem} \label{theorem:maximal_locally_idempotent}
Let $f$ be a locally idempotent network with interaction graph $D$ and let $f^\land$ and $f^\lor$ be the conjunctive and disjunctive networks on $D$, respectively. Then 
\[
    f^\land \le f \le f^\lor.
\]
\end{theorem}

\begin{proof}
For all $x \in \B^n$ and all $i \in [n]$,
\begin{align*}
    f_i(x) &= 0 \implies x_{\inNeighbourhood(i)} \ne (1, \dots, 1) \implies f^\land_i(x) = 0,\\
    f_i(x) &= 1 \implies x_{\inNeighbourhood(i)} \ne (0, \dots, 0) \implies f^\lor_i(x) = 1.
\end{align*}
This yields $f^\land_i(x) \le f_i(x) \le f_i^\lor(x)$.
\end{proof}

The fourth characterisation of disjunctive networks is to be ``closest'' to constant networks--technically this does not fully characterise disjunctive networks, as any network where the local function is a junction satisfies this property. The \Define{distance} between two Boolean networks $f,g : \B^n \to \B^n$ is
\[
    d(f, g) := \sum_{x \in \B^n} \dH(f(x), g(x)).
\]
A Boolean network $g$ is \Define{constant} if there exists $c \in \B^n$ for which $g(x) = c$ for all $x \in \B^n$; we denote the set of constant networks as $C$. The \Define{distance to constant networks} of $f$ is then defined as
\begin{align*}
    \dc( f ) &:= \min\{ d(f, g) :  g \in C \}\\
    &= \min_{c \in \B^n} \sum_{x \in \B^n} \dH(f(x), c).
\end{align*}
Then $\dc(f)$ ranges from $0$ (whenever $f$ is constant) to $n 2^{n - 1}$ (whenever $f$ is a permutation, amongst others).

\begin{theorem} \label{theorem:minimise_distance_to_constant}
Let $f$ be a Boolean network with interaction graph $D$ and let $f^\land$ and $f^\lor$ be the conjunctive and disjunctive networks on $D$, respectively. Then 
\[
    \dc(f) \ge \dc(f^\lor) = \dc(f^\land).
\]
\end{theorem}

\begin{proof}
For any $a,b \in \B$, we denote $a \oplus b = a + b \mod 2$, i.e. $a \oplus b = 1$ if and only if $a = \neg b$.  We have
\begin{align*}
    d_C(f) &= \min_{c \in \B^n} \sum_{x \in \B^n} \sum_{i \in [n]} f_i(x) \oplus c_i\\
    &= \min_{c \in \B^n} \sum_{i \in [n]} |f_i^{-1}(\neg c_i)|.
\end{align*}
For any $i$, let $d_i$ denote the in-degree of $i$ in $D$; let $T$ denote the set of non-sources of $D$. We have
\begin{align*}
    \dc(f) &= \min_{c \in \B^n} \sum_{i \in [n]} | f_i^{-1}( \neg c_i) |\\
    &=  \sum_{i \in [n]} \min\{ |f_i^{-1}(0)|, |f_i^{-1}(1)| \}\\
    &= \sum_{i \in T} \min\{ |f_i^{-1}(0)|, |f_i^{-1}(1)| \}\\
    &\ge \sum_{i \in T} 2^{n- d_i}.
\end{align*}
It is clear that the last inequality is reached for the disjunctive (or conjunctive) network on $D$.
\end{proof}

\subsection{Representations} \label{section:representations}

Let $f$ be the disjunctive network on $D = (V = [n], E)$. We give below four representations of $f$. 
\begin{description}
\item[Boolean linear mapping] A graph can be represented by its adjacency matrix $A_D$ where $a_{i,j} = 1$ if and only if $(i,j) \in E$. The product of two Boolean matrices is $AB = C$ where 
\[
    c_{ij} = \bigvee_{k=1}^n a_{ik} \land b_{kj}.
\]
Boolean matrix theory has been widely studied from an algebraic and combinatorial point of view, and has found applications in different areas of computer science; the interested reader is directed to the authoritative book \cite{Kim82}.  A \Define{Boolean linear mapping} is any $g : \B^n \to \B^n$ of the form $g(x) = x A$ for some Boolean matrix $A$. In our case, the disjunctive network on $D$ satisfies $f(x) = x A_D$. Since any Boolean matrix is the adjacency matrix of some graph, disjunctive networks are exactly the Boolean linear mappings. We remark that Boolean linear mappings are different from their finite field counterparts, which are usually simpler to analyse.

\item[Out-neighbourhood function] Identifying $x \in \B^n$ with its support  $X = \supp(x) = \{ i \in [n] : x_i = 1 \}$, we can identify $f$ with the mapping on the power set of $[n]$ defined by
\[
    f(X) = \outNeighbourhood(X).
\]

\item[Binary relation] A \Define{binary relation} $R$ on $[n]$ is a subset of $[n] \times [n]$. Binary relations are in one-to-one correspondence with Boolean matrices; as such, the semigroup of binary relations has been widely studied \cite{How95}. Clearly graphs are also in one-to-one correspondence with binary relations (the edge set $E$ of $D$ is a binary relation). We can then represent $f$ as $f(X) = \{ y \in [n] : \exists s \in X, (s,y) \in E \}$.

\item[Token sliding] We can rewrite the representation above as
\[
    f_i(X) = \bigcup_{j \in \inNeighbourhood (i)} X_j.
\]
This can be interpreted as follows. Suppose there are $n$ tokens $T = \{t_1, \dots, t_n\}$ on the $n$ vertices of $D$. Let $X_i \subseteq T$ denote the collection of tokens that $i$ holds at a given time. At each time step every vertex $i$ broadcasts all its tokens to all possible destinations. Then vertex $j$ obtains all tokens from its in-neighbours, and hence $X_j$ becomes $\bigcup_{j \in \inNeighbourhood (i)} X_j$.

\end{description}

\section{Dynamical properties} \label{section:dynamic_properties}

Let $f$ be a Boolean network. We consider three types of points for $f$:
\begin{itemize}
    \item 
    An \Define{image point} of $f$ is $x \in \B^n$ such that $f(y) = x$ for some $y \in \B^n$. In the transformation semigroup literature, the number of image points is usually called ``rank''. However, different ranks for the matrix $A_D$ have been proposed \cite{Kim82}. Therefore, in order to emphasize that we are counting the number of image points, we shall use refer to the number of image points of $f$ as its \Define{image rank}.
    
    \item
    A \Define{periodic point} of $f$ is $x \in \B^n$ such that $f^k(x) = x$ for some $k \ge 1$. The number of periodic points of $f$ is called the \Define{periodic rank} of $f$. Clearly, the periodic rank of $f$ is equal to the image rank of $f^{2^n}$, and is less than or equal to the image rank of $f^p$ for any $p \ge 1$.
    
    \item
    A \Define{fixed point} of $f$ is $x \in \B^n$ such that $f(x) = x$. The number of fixed points of $f$ is called the \Define{fixed rank} of $f$. Unlike the image and periodic ranks, the fixed rank of a Boolean network can be equal to zero.
\end{itemize}

\subsection{Image, periodic and fixed points} \label{subsection:points}

In this subsection, we review some of the key results describing the sets of image points, periodic points, and fixed points of disjunctive networks. It will be convenient to identify a Boolean configuration $x \in \B^n$ with its support $X = \{ i \in [n] : x_i = 1 \} \subseteq [n]$ and to view the disjunctive network on $D$ as $f : X \mapsto \outNeighbourhood(X)$.

\subsubsection{Image points} \label{section:image_points}

Given a disjunctive network $f$ and a subset $X$, it is easy to verify whether $X$ is an image point of $f$. By definition, $X$ belongs to the image of $f$ if it is the out-neighbourhood of some set of vertices: $X = \outNeighbourhood(Y)$ for some $Y \subseteq [n]$. The key property is that there is a unique maximal preimage $Y^*$ of $X$, so we only need to compute $f(Y^*)$ and check whether it matches with $X$.

\begin{proposition}
For any $X \subseteq [n]$, let 
\[
    Y^* := [n] \setminus \left( \inNeighbourhood( [n] \setminus  X  ) \right) .
\]
Then $X$ is an image point of $f$ if and only if $X = \outNeighbourhood(Y^*)$. If that is the case, then 
\[
    Y^* = \bigcup_{Y \in f^{-1}(X)} Y.
\]
\end{proposition}

\begin{proof}
Suppose $X = f(Y)$ for some $Y$. Firstly, $Y \subseteq Y^*$, for if $j \in Y \cap \inNeighbourhood(i)$ for some $i \notin X$, then $i \in \outNeighbourhood(Y)$. Therefore, $X \subseteq f(Y^*)$. Conversely, $\outNeighbourhood(Y^*) \subseteq X$. Indeed, for any $j \notin X$, $\inNeighbourhood(j) \cap Y^* = \emptyset$ and hence $j \notin \outNeighbourhood(Y^*)$. Combining, we obtain $X = f(Y^*)$.

The second statement follows from the fact that $f(Y) = X$ implies $Y \subseteq Y^*$.
\end{proof}

Determining the image rank of a function of the form $f(x) = xM$, where $M$ is over a finite field $\GF(q)$, is simple: it is given by $q^{\mathrm{rank}(M)}$, where the rank can be computed in polynomial time. On the other hand, it is not even clear whether computing the image rank of a disjunctive network can be done in polynomial time.

\begin{problem}
What is the complexity of the following problem: given a graph $D$, what is the image rank of the disjunctive network on $D$?
\end{problem}

\subsubsection{Periodic points} \label{section:periodic_points}

The set of periodic points of the disjunctive network $f$ on $D$ have been studied in \cite{GH00, SLV10, AMPV12}. In order to keep the notation simple and to give the intuition behind the main results, we focus on the case where $D$ is \Define{strong} (a.k.a. strongly connected), i.e. there is a path from $i$ to $j$ for any two vertices $i$ and $j$. For results about general graphs, see \cite{SLV10}.  

The \Define{loop number} $l(D)$ of $D$ is the greatest common denominator of all the cycle lengths in $D$. If $l(D) = 1$ and $D$ is strong, we say that $D$ is \Define{primitive} (graphs with loop number one are sometimes called aperiodic). It is well known that if $l(D) = 2$, then the graph is bipartite: we can partition its vertex set in two parts such that all the arcs go between the parts. This is generalised as follows.

\begin{lemma} \cite[Theorem 17.8.1]{BG09a}
If a strong digraph $D = (V,E)$ has loop number $l(D) = l \ge 2$, then $V$ can be partitioned into $V_0, \dots, V_{l-1}$ such that $\outNeighbourhood(V_i) = V_{i+1}$ (indices computed modulo $l$).
\end{lemma}

Say a subset of vertices $X$ is \Define{$D$-partite} if 
\[
    X = V_{i_1} \cup \dots \cup V_{i_a}
\]
for some $i_1, \dots, i_a \in [l]$. Clearly, $X$ is a periodic point, since $f^l(X) = X$. The key result is that a subset is a periodic point if and only if it is $D$-partite. Indeed, let $Y$ be non-$D$-partite and let $s \in Y \cap V_i$ and $t \in V_i \setminus Y$. There is a path from $s$ to $t$; that path must have length equal to $k l$ for some $k$, hence $t \in f^{kl}(Y)$. Generalising this idea, we obtain that $f^q(Y)$ will eventually be $D$-partite for $q$ large enough.

\begin{theorem}\cite{GH00}
Let $D$ be a strong graph, then a subset $X$ of vertices is a periodic point of the disjunctive network $f$ on $D$ if and only if $X$ is $D$-partite.
\end{theorem}

The \Define{period} of a periodic point $X$ is the smallest $p \ge 1$ such that $f^p(X) = X$. We have seen that $f^{l(D)}(X) = X$, which immediately yields:

\begin{corollary}\cite{GH00,SLV10}
The period of a periodic point of $f$ divides $l(D)$. Conversely, for any $p \ | \ l(D)$, there is a periodic point of period $p$.
\end{corollary}

In particular, if $D$ is primitive, then the only periodic points are $\emptyset$ and $[n]$, which are fixed points. More results can be obtained, for instance the time it takes to reach a periodic points can be upper bounded \cite{GH00} and the number of periodic points of a certain period can be determined \cite{SLV10}.

\subsubsection{Fixed points} \label{section:fixed_points}

We give a classification of the set of fixed points of a disjunctive network. In order to keep it simple, we focus on \Define{nontrivial} graphs, where every vertex belongs to a cycle.

The axioms of topology simplify greatly for finite spaces: a collection $T$ of subsets of $[n]$ is a \Define{topology} on $[n]$ if and only if $\emptyset, [n] \in T$ and
\[
    X,Y \in T \implies X \cup Y, X \cap Y \in T.
\]
Let $D = ([n],E)$ be a nontrivial graph and for any $i,j \in [n]$ denote $i \le j$ if there is a path from $i$ to $j$. We note that $\le$ is only a preorder relation, and that reflexivity is guaranteed by the fact that $D$ is nontrivial. For any $S \subseteq [n]$, let the \Define{up-set} of $S$ to be the set of vertices reachable from $S$:
\[
    S^\uparrow := \{ j : \exists i \in S, i \le j \}.
\]
It is easily checked that the collection of up-sets forms a topology. We denote this collection as
\[
    \Topology(D) := \{ S^\uparrow : S \subseteq [n] \}.
\]

\begin{theorem} \label{theorem:topology}
The set of fixed points of the disjunctive network on the nontrivial graph $D$ is $\Topology(D)$.
\end{theorem}

\begin{proof}
The subset $X$ is a fixed point of $f$ if and only if $X = \outNeighbourhood(X)$, that is $X = S^\uparrow$ for some $S$.
\end{proof}

In fact, any finite topology arises from a nontrivial graph; this well known result is usually given in terms of preorders (see \cite{Cam94}). 

\begin{lemma} \label{lemma:topology} \cite[Theorem 3.9.1]{Cam94}
Let $T$ be a topology on $[n]$, then $T = \Topology(D)$ for some nontrivial graph $D$ on $[n]$.
\end{lemma}

The Knaster-Tarski theorem asserts that the set of fixed points of a monotone network forms a lattice. Conversely, for any lattice of configurations of $\B^n$, it is easy to construct a monotone network whose set of fixed points is exactly that lattice. Therefore, a subset of $\B^n$ is the set of fixed points of a monotone network if and only if it forms a lattice. We obtain a similar characterisation for the sets of fixed points of disjunctive networks.

\begin{corollary}  \label{corollary:topology}
$T$ is the set of fixed points of a disjunctive network on a nontrivial graph if and only if it is a topology.
\end{corollary}

\subsection{Values of the image, periodic and fixed ranks} \label{subsection:ranks}

In this subsection, we consider the values that can be taken by the different ranks of a disjunctive network. Foremost, a Boolean network is \Define{idempotent} if $f^2 = f$, i.e. $f(f(x)) = f(x)$ for all $x \in \B^n$. It is clear that a Boolean network is idempotent if and only if its fixed rank, periodic rank and image rank are all equal. Idempotent disjunctive networks were characterised, under the guise of binary relations, by Rosenblatt \cite{Ros63} (see \cite{Kim72}).

The values the image/periodic/fixed rank can take for a monotone network are easy to characterise. Note that the Knaster-Tarski theorem implies that any monotone network has at least one fixed point.

\begin{proposition}
For any $n$ and any $1 \le k \le 2^n$, there exists an idempotent monotone network of dimension $n$ with exactly $k$ image points. 
\end{proposition}

\begin{proof}
Sort the configurations of $\B^n$ in non-decreasing order of Hamming weight, so that 
$x^0 = (0, \dots, 0), \dots, x^{2^n-1} = (1, \dots, 1)$ and $x^i \le x^j$ implies $i \le j$. Let $f$ be defined as
\[
    f(x^a) = \begin{cases}
    x^a         &\text{if } 0 \le a \le k-2\\
    x^{2^n-1}   &\text{if } k-1 \le a \le 2^{n-1}.
    \end{cases}
\]
Then $f$ is idempotent, monotone and has image/periodic/fixed rank $k$. 
\end{proof}

On the other hand, the situation for disjunctive networks is more complex. We shall obtain some results on the values the image/periodic/fixed rank of a disjunctive network can take but we will remain far from classifying them.

\begin{problem}
What is the complexity of the following problem: given $k$ and $n$, is there a disjunctive network of dimension $n$ with image rank $k$? 
\end{problem}

First, we consider the maximum value of the image/periodic/fixed ranks. Clearly, it is given by $2^n$, but we can even classify the disjunctive networks with image/periodic rank $2^n$. (There is only one Boolean network with fixed rank $2^n$, namely the identity, which happens to be disjunctive.) In fact, the only bijective monotone networks are \Define{permutations of variables}. This classification is folklore, and is related to the classification of isometries of the hypercube. The symmetric group on $[n]$ acts naturally on configurations in $\B^n$, whereby  $\pi(x) = ( x_{\pi(1)}, \dots, x_{\pi(n)} )$ for any permutation $\pi$ of $[n]$. Any such Boolean network is called a permutation of variables. Note that permutations of variables are disjunctive networks: they correspond to interaction graphs that are disjoint unions of cycles.

\begin{theorem} \label{theorem:permutation}
A monotone Boolean network is bijective if and only if it is a permutation of variables.
\end{theorem}

\begin{proof}
Let $f$ be a bijective monotone Boolean network. We first prove that $f$ preserves the Hamming weight. Any $x \in \B^n$ of Hamming weight $k$ belongs to a maximal chain 
\[
    x^0 = (0, \dots, 0) < x^1 < \dots < x^k = x < \dots < x^n = (1, \dots, 1),
\]
where $\wH(x^i) = i$ for all $0 \le i \le n$. By monotonicity and injectivity, we have
\[
    f(x^0) < f(x^1) < \dots < f(x^n),
\]
whence $\wH(f(x^i)) = i$ for all $0 \le i \le n$, and in particular $\wH(f(x)) = x$. 

We now prove that $f(x) = \pi(x)$ for some permutation $\pi$, by induction on $k = \wH(x)$. The base cases $k = 0$ and $k=1$ are clear, so let us assume $k \ge 2$ and that it holds for up to $k-1$. Suppose $x_i = x_j = 1$, and denote $x^{-i} = (0, x_{-i})$ and $x^{-j} = (0, x_{-j})$. We have
\[
    f(x) \ge f(x^{-i}) \lor f(x^{-j}) = \pi(x^{-i}) \lor \pi(x^{-j}) = \pi(x).
\]
Since $f(x)$ and $\pi(x)$ both have Hamming weight $k$, we have equality: $f(x) = \pi(x)$.
\end{proof}

We now move on to singular networks, i.e. those that are not bijective. Even though there are monotone Boolean networks of image rank $k$ for all $1 \le k \le 2^n$, and in particular for $k = 2^n - 1$, the image rank of singular disjunctive networks is upper bounded by $3/4 \cdot 2^n$. We fully classify the graphs that reach the upper bound in Theorem \ref{theorem:near-bijective} below. The classification is based on the following three families of connected graphs:
\begin{enumerate}
    \item
    $C_n$ ($n \ge 1$) is the cycle on $n$ vertices, with vertex set $V = \Z_n$ and arcs $E = \{ (i, i+1 \mod n) : i \in \Z_n \}$ for $n \ge 2$ and $E = \{(0,0)\}$ for $n=1$.

    \item
    $A_{p,q}$ ($0 \le q \le p-1$) is the chorded cycle: the cycle $C_p$, augmented by the chord $\{(0, q)\}$.
    
    \item
    $B_{s,t}$ ($s,t \ge 1$) is the link of cycles: formed of two cycles $C_s$ and $C_t$, with a single arc from $C_s$ to $C_t$. We denote the vertices of $C_s$ as $0, \dots, s-1$ and those of $C_t$ as $\overline{0}, \dots, \overline{t-1}$.
\end{enumerate}
Examples of these graphs ($C_4$, $A_{6,2}$ and $B_{1,3}$) are displayed on Figure \ref{figure:graphs} . We then say that a graph is \Define{near-cyclic} if one of its connected components is a chorded cycle or a link of cycles, and all other connected components are cycles. More formally, $D$ is near-cyclic if it is of the form $D = A_{p,q} \cup C_{n_1} \cup \dots \cup C_{n_c}$ or $D = B_{s,t} \cup C_{n_1} \cup \dots \cup C_{n_c}$ for some choice of parameters.

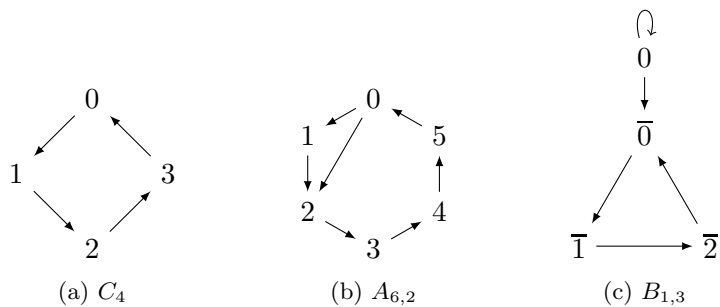
\begin{figure}
    \centering
    \subfloat[$C_4$]{
    \begin{tikzpicture}
        
        \foreach \x in {0,...,3}{
            \node (\x) at (90 + \x * 90:1) {\x};
        }

        \draw[-latex] (0) edge (1) (1) edge (2) (2) edge (3) (3) edge (0);
    \end{tikzpicture}
    } \hspace{1cm}
    \subfloat[$A_{6,2}$]{
    \begin{tikzpicture}
        \foreach \x in {0,...,5}{
            \node (\x) at (90 + \x * 60:1) {\x};
        }
        
        \draw[-latex] (0) edge (1) (1) edge (2) (2) edge (3) (3) edge (4) (4) edge (5) (5) edge (0) (0) edge (2);
            
    \end{tikzpicture}
    } \hspace{1cm}
    \subfloat[$B_{1,3}$]{
    \begin{tikzpicture}
        \foreach \x in {0,1,2}{
            \node (\x) at (90 + \x * 120:1) {$\overline{\x}$};
        }
        
        \node (v) at (0,2) {$0$};
        \draw[-latex] (v) edge [loop above] (v) (v) edge (0);
        
        \draw[-latex] (0) edge (1) (1) edge (2) (2) edge (0);
            
    \end{tikzpicture}
    }
    \caption{Graphs used in Theorem \ref{theorem:near-bijective}}
    \label{figure:graphs}
\end{figure}

\begin{theorem} \label{theorem:near-bijective}
For $n \ge 2$, the maximum image rank of a singular disjunctive network of dimension $n$ is $3/4 \cdot n$, and it is reached if and only if its interaction graph is near-cyclic.
\end{theorem}

\begin{proof}
We first prove that the image rank is upper bounded by $3/4 \cdot 2^n$. Let $f$ be a singular disjunctive network and $D$ be its interaction graph. 

If all the vertices of $D$ have in-degree one, then it must have a sink $k$ (otherwise, $f$ would be a permutation of variables). But then, any two configurations only differing in position $k$ have the same image under $f$, and hence $|f(\B^n)| \le \frac{1}{2} 2^n$.

If $D$ has a vertex $i$ of in-degree $d \ge 2$, then $|f_i^{-1}(0)| = 2^{n-d}$, and hence
\begin{align*}
    | f(\B^n) | &\le |\{ y \in \B^n : y_i = 1 \}| + |f_i^{-1}(0)|\\
    &\le 2^{n-1} + 2^{n-d}\\
    &\le 3/4 \cdot 2^n.
\end{align*}

If $D$ has a source $i$, then $|f_i^{-1}(1)| = 0$, and we similarly obtain $|f(\B^n)| \le 1/2 \cdot 2^n$.

We now characterise the graphs that reach the upper bound. By the above, there is a vertex of in-degree $2$. We first prove that it is unique. Suppose for the sake of contradiction that both $i$ and $j$ have in-degree $2$. We do a case analysis based on $|\inNeighbourhood(i,j)|$. For the sake of simplicity, we write $\{ ab = 01, 10 \}$ as a shorthand for $\{ y \in \B^n : y_a y_b \in \{01, 10\} \}$ and we extend this shorthand notation to arbitrary sets of coordinates and values.
\begin{itemize}
    \item 
    $|\inNeighbourhood(i,j)| = 2$. Then $f_i = f_j$ and hence 
    \[
        |f(\B^n)| \le |\{ ij = 00, 11 \}| = \frac{1}{2} 2^n.
    \]
    
    \item
    $|\inNeighbourhood(i,j)| = 3$. Say $\inNeighbourhood(i) = \{a,b\}$ and $\inNeighbourhood(j) = \{a,c\}$. Then 
    \begin{align*}
        |f(\B^n)| &= |f(\{ abc = 000,001,010 \})| + |f( \{ abc = 011, 100, 101, 110, 111  \})| \\
        &\le |\{ abc = 000,001,010  \}| + |\{ ij = 11 \}|\\
        &= \frac{5}{8} 2^n. 
    \end{align*}
    
    \item
    $|\inNeighbourhood(i,j)| = 4$. Say $\inNeighbourhood(i) = \{a,b\}$ and $\inNeighbourhood(j) = \{c,d\}$. Then 
    \begin{align*}
        |f(\B^n)| &= |f(\{ abcd = 0000,0001,0010, 0011, 0100, 1000, 1100 \})| \\
        & \phantom{= } + |f(\{ abcd = 0101, 0110, 0111, 1001, 1010, 1011, 1101, 1110, 1111 \})|\\
        &\le  |\{ abcd = 0000,0001,0010, 0011, 0100, 1000, 1100 \}| + |\{ ij = 11 \}| \\
        &= \frac{11}{16} 2^n. 
    \end{align*}
\end{itemize}

We can now prove the result for reflexive graphs. A graph is \Define{reflexive} if $(i,i) \in E$ for all $i \in V$. The only reflexive graph with a unique vertex of in-degree $2$ is $G_n = B_{1,1} \cup C_1 \cup \dots \cup C_1$, displayed on Figure \ref{figure:G}; it is easy to check that the disjunctive network on $G_n$ indeed has image rank $3/4 \cdot 2^n$.

Let $G = (V,E)$ be a graph and $\pi$ be a permutation of $V$. We define $D = G \pi$ the graph with vertex set $V$ and arc set $\{ (i,\pi(j)) : (i,j) \in E \}$. In general, if a graph $D$ admits a Boolean network with image rank greater than $1/2 \cdot 2^n$, then it must be coverable by cycles \cite{Gad18a}, i.e. there exists a permutation $\pi$ of $[n]$ such that $D = G \pi$, where $G$ is reflexive. Thus, if $D$ is the interaction graph of a disjunctive network of image rank $3/4 \cdot 2^n$, then it is of the form $D = G_n \pi$ for some permutation $\pi$. If $\pi(1)$ and $\pi(2)$ belong to the same cycle of $\pi$, we obtain $D = A_{p,q} \cup C_{n_1} \cup \dots \cup C_{n_c}$; otherwise we obtain $D = B_{s,t} \cup C_{n_1} \cup \dots \cup C_{n_c}$.

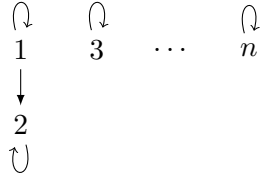
\begin{figure}
    \centering
    \begin{tikzpicture}
        \node (1) at (0,1) {$1$};
        \node (2) at (0,0) {$2$};
        \node (3) at (1,1) {$3$};
        \node (4) at (2,1) {$\dots$};
        \node (n) at (3,1) {$n$};
        
        \draw[-latex] (1) edge [loop above] (1)
        (2) edge [loop below] (2)
        (3) edge [loop above] (3)
        (n) edge [loop above] (n)
        (1) edge (2);
 
    \end{tikzpicture}
    \caption{The graph $G_n$}
    \label{figure:G}
\end{figure}

\end{proof}

We also obtain the corresponding result for the periodic rank.

\begin{corollary} \label{corollary:periodic_points}
For $n \ge 2$, the maximum periodic rank of a singular disjunctive network of dimension $n$ is $3/4 \cdot 2^n$, and reached if and only if the interaction graph is of the form $B_{1,1} \cup C_{n_1} \cup \dots \cup C_{n_c}$.
\end{corollary}

\begin{proof}
Let $f$ be a disjunctive network on the near-cyclic graph $D$, and let $D^2$ denote the interaction graph of $f^2$. If $D$ has an $A_{p,q}$ component, then $D^2$ has two vertices of in-degree at least two, namely $q$ and $q+1$. A similar argument holds for a $B_{s,t}$ component with $t \ge 2$. For a $B_{s,1}$ component with $s \ge 2$, the vertex $\overline{0}$ has in-degree $3$ in $D^2$. Therefore, in any case, the image rank of $f^2$ is less than $3/4 \cdot 2^n$, and so is the periodic rank of $f$.
\end{proof}

The corresponding result for the fixed rank immediately follows.

\begin{corollary}
For $n \ge 2$, the maximum fixed rank of a singular disjunctive network of dimension $n$ is $3/4 \cdot 2^n$, and reached if and only if the interaction graph is $G_n$.
\end{corollary}

\begin{proof}
If a disjunctive network is singular, then its fixed rank is at most $3/4 \cdot 2^n$ by Corollary \ref{corollary:periodic_points}. Therefore, we can restrict ourselves to disjunctive networks on near-cyclic graphs of the form $B_{1,1} \cup C_{n_1} \cup \dots \cup C_{n_c}$. It is clear that if any cycle $C_{n_i}$ has $n_i \ge 2$, then the disjunctive network has periodic points that are not fixed points; conversely the disjunctive network on $G_n$ is idempotent.
\end{proof}

We next consider the opposite problem: what is the smallest ``missing value'' of the image/periodic/fixed rank? Obviously, we consider a nonzero fixed rank.

\begin{problem}
Given $n$, what is the minimum $k \ge 1$ such that there is no disjunctive network of dimension $n$ and image/periodic/fixed rank $k$?
\end{problem}

We give a lower bound on that quantity below, for all three ranks.

\begin{theorem} \label{theorem:low_rank}
For any $n$, there exists an idempotent disjunctive network on $n$ vertices with image, periodic, and fixed ranks $r$ for all $1 \le r \le p - 1$, where $p$ is the smallest prime number greater than $n + 1$.
\end{theorem}

\begin{proof}
The reflexive transitive tournament $T_a$ on $a$ vertices has arcs $ij$ for all $i \le j$. Let $f$ be the disjunctive network on $T_a$, then it is easily seen that $f$ is idempotent and that its image is $\{ \bigvee_{i=j}^a e^i : 1 \le j \le a+1 \}$.

We can now prove the result; it is clear for $n \le 2$ so we suppose $n \ge 3$. First, suppose $r \le n+1$. We denote the empty graph on $c$ vertices as $E_c$. Then the disjunctive network on $T_{r-1} \cup E_{n-r+1}$ is idempotent and has image rank $r$. Second, suppose $n + 2 \le r \le p-1$. Then by Bertrand's postulate, $p \le 2n - 1$ and hence $r \le 2n - 2$. Since $r$ is composite, we have $r = ab$ for $a + b \le 2 + r/2 \le n + 1$. Thus the disjunctive network on the graph $T_{a-1} \cup T_{b-1} \cup E_{n-a-b+2}$ is idempotent and has image rank $r$.
\end{proof}

\section{Outlook} \label{sec:conclusion}

\subsection{Disjunctive networks compared to other networks}


It would be interesting to compare the disjunctive network on $D$ with the other Boolean networks on $D$. Firstly, we want to investigate when the disjunctive network has as many image/periodic/fixed points as possible. There are three main upper bounds on the image/periodic/fixed rank of a Boolean network with interaction graph $D$, that depend on three graph parameters reviewed in \cite{Gad20}. The image rank of a Boolean network $f$ on $D$ is at most $2^{\alpha_1(D)}$ \cite{Gad18a}, while its periodic rank is at most $2^{\alpha_n(D)}$ \cite{Gad18a}, and its fixed rank is at most $2^{\tau(D)}$ (the famous feedback bound \cite{Rii06,Ara08}). Those bounds are not always reached (e.g. the pentagon does not reach any). The graphs where the feedback bound is reached by a monotone network are classified in \cite{ARS17}. Similar classification results for disjunctive networks seem close at hand.

\begin{problem}
Classify the graphs $D$ such that:
\begin{enumerate}
    \item the image rank of the disjunctive network on $D$ is $2^{\alpha_1(D)}$.
    
    \item the periodic rank of the disjunctive network on $D$ is $2^{\alpha_n(D)}$.

    \item the fixed rank of the disjunctive network on $D$ is $2^{\tau(D)}$.
\end{enumerate}
\end{problem}

Secondly, we want to investigate when the disjunctive network minimises the image/periodic/fixed rank. Since disjunctive networks are the closest to being constant by Theorem \ref{theorem:minimise_distance_to_constant}, one might expect that they should minimise the image rank over all networks with a given interaction graph. This is true in the following extreme cases, where the minimum rank is equal to $1$, $2$, or $2^n$ \cite{Gad20}. However, this turns out not to be the case in general: \cite[Theorem 5]{Gad20} gives a counter-example, where the minimum image rank is not achieved by the disjunctive network, but can be actually achieved by another monotone network. We therefore ask the following two questions.

\begin{problem}
Classify the graphs $D$ such that the disjunctive network on $D$ minimises the image rank over
\begin{enumerate}
    \item 
     all Boolean networks on $D$;
    
    \item
    all monotone networks on $D$. 
\end{enumerate}
\end{problem}

For periodic points, we do not know which graphs $D$ admit a Boolean network with a single periodic point (the so-called nilpotent networks) \cite{GR16}. As such it seems difficult to characterise the graphs where the disjunctive network minimises the periodic rank. For fixed points, on the other hand, the problem is straightforward. If $D$ is acyclic, then all the Boolean networks on $D$ have a unique fixed point by Robert's celebrated theorem \cite{Rob80}. Conversely, Aracena and Salinas (private communication) showed that any non-acyclic graph $D$ admits a fixed point free Boolean network.

\subsection{Generalisations}

We mention three possible avenues of generalising the scope of the current study of disjunctive networks. For each, it would be interesting to see how the results presented here might generalise to these broader classes of networks.
\begin{description}
    \item[Other Boolean networks] There are many classes of Boolean networks that contain, or are closely related to, disjunctive networks, e.g. AND-OR networks \cite{ADG04a}, AND-OR-NOT networks \cite{ARS14}, nested canalyzing networks \cite{KPST04}, threshold networks \cite{GO81,Gol85}, and of course monotone networks \cite{ADG04b}. 
    
    \item[Higher alphabets] The disjunction can be easily generalised to variables taking their values over a linearly ordered alphabet, by taking the maximum \cite{AMPV12}. This is not the only choice that maintains some of the desirable properties of Boolean disjunction; in fact, a thorough study and classification of disjunction functions in multiple valued logics is given in \cite{Got01}. 
    A different approach views the conjunction as the product: $x \land y = xy$ for all $x,y \in \B$ \cite{CLP05}. One can then consider so-called monomial dynamical systems over finite fields \cite{CJLS06}, where each local function is a product of variables; note that those networks are not monotone.
    
    \item[Infinite graphs] One can define the disjunction of any set of Boolean variables, whether finite or infinite. Considering disjunctive networks over infinite graphs would be another step to link Boolean networks and Cellular automata.
\end{description}


\end{document}